\documentclass[12pt, reqno]{amsart}
\makeatletter
\@namedef{subjclassname@1991}{$\mathrm{1991}$ Mathematics Subject Classification}
\@namedef{subjclassname@2000}{$\mathrm{2000}$ Mathematics Subject Classification}
\@namedef{subjclassname@2010}{$\mathrm{2010}$ Mathematics Subject Classification}
\@namedef{subjclassname@2020}{$\mathrm{2020}$ Mathematics Subject Classification}
\makeatother
\usepackage{amsmath,amsthm, amscd, amsfonts, amssymb, graphicx, color}
\usepackage[bookmarksnumbered, colorlinks, plainpages,linkcolor=blue,urlcolor=blue,citecolor=blue]{hyperref}
\newtheorem{thm}{Theorem}[section]
\newtheorem{cor}[thm]{Corollary}
\newtheorem{lem}[thm]{Lemma}

\newtheorem{defn}[thm]{Definition}

\newtheorem{exam}[thm]{Example}
\numberwithin{equation}{section}

\begin{document}

\title{generalized $[v,w]$-core inverse in Banach *-algebras}

\author{Huanyin Chen}
\address{School of Big Data, Fuzhou University of International Studies and Trade, Fuzhou 350202, China}
\email{<huanyinchenfz@163.com>}

\subjclass[2020]{16U90, 15A09, 46H05.} \keywords{core inverse; weighted core inverse; generalized weighted core inverse; core-EP inverse; weighted g-Drazin inverse; Banach algebra.}

\begin{abstract} We introduce the generalized weighted core inverse with respect to two elements in a Banach *-algebra, extending the framework of existing generalized inverses. This new generalized inverse is characterized via the generalized weighted core decomposition, and its representations are established using the weighted g-Drazin inverse. We investigate the polar-like and principal ideal properties to reveal its structural behavior. Consequently, applications to complex matrices are derived. This finds encompasses both the core-EP and generalized core inverses, significantly broadening the theoretical framework.\end{abstract}

\maketitle

\section{Introduction}

A Banach algebra $\mathcal{A}$ is called a Banach *-algebra if there exists an involution $*: x\to x^*$ satisfying $(x+y)^*=x^*+y^*,
(\lambda x)^*=\overline{\lambda} x^*, (xy)^*=y^*x^*, (x^*)^*=x$. An element $a$ in a Banach *-algebra $\mathcal{A}$ has core inverse if there exists an element
$x\in \mathcal{A}$ such that $$ax^2=x, (ax)^*=ax, xa^2=a.$$ If such $x$ exists, it is unique, and denote it by $a^{\tiny\textcircled{\#}}$ (see~\cite{CZ2}).

Gao and Chen extended the concept of the core inverse and introduced the notion of core-EP inverse (i.e., pseudo core inverse) (see~\cite{GC}). An element $a\in \mathcal{A}$ has core-EP inverse if there exist $x\in \mathcal{A}$ and $k\in \Bbb{N}$ such that $$ax^2=x, (ax)^*=ax, xa^{k+1}=a^k.$$ If such $x$ exists, it is unique, and denote it by $a^{\tiny\textcircled{D}}$. The core and core-EP inverses have been studied by many authors, e.g., ~\cite{CCZ,GC1,M3,M5,R,XS,Z}.

To extend core and core-EP inverses to the general case, $w$-core and
$w$-core-EP inverses were introduced in ~\cite{M6,ZW}. For further research on these weighted generalized inverses, we refer the reader to ~\cite{B,B1,D,JZ,MD,M4,M6,W,Y,ZC}.

More recently, Mary et al. generalized the $w$-core-EP inverse and introduced a new type of generalized inverse. An element $a\in R$ has $(v,w)$-pseudo-core inverse if there exists an element $x\in R$ such that $$awxvx=x, (vawx)^*=vawx,xv(aw)^{k+1}a=(aw)^ka$$ for some $k\in {\Bbb N}$. If such $x$ exists, it is unique, and denote it by $a_{[v,w]}^{\tiny\textcircled{D}}$. Let $\mathcal{A}_{[v,w]}^{\tiny\textcircled{D}}$ denote the set of all $(v,w)$-pseodu-core invertible elements in $\mathcal{A}$ (see~\cite{M}).

To extend the $w$-core inverse to a more general setting, we introduce the following new concept:

\begin{defn} Let $a,w\in \mathcal{A}$ and $v\in \mathcal{A}^{-1}$. An element $a\in \mathcal{A}$ has $[v,w]$-core inverse if there exists $x\in \mathcal{A}$ such that
$$awxvx=x, xvawa=a, (vawx)^*=vawx.$$\end{defn}

Let $a_{[v,w]}^{\tiny\textcircled{\#}}$ denote the preceding $x$, and call it the $[v,w]$-core inverse of $a$.
Obviously, the concept of the $[v,w]$-core inverse naturally extends the known weighted core inverse, encompassing both the
$w$-core inverse and the weighted $w$-core inverse with the weight $v$ as special cases (see~\cite{W,ZW}).

Let $\mathcal{A}_{[v,w]}^{\tiny\textcircled{\#}}$ denote the set of all $[v,w]$-core invertible element.
Let $$\mathcal{A}^{qnil}=\{x\in \mathcal{A}~\mid~ \lim\limits_{n\to \infty}\parallel x^n\parallel^{\frac{1}{n}}=0\}.$$ As it is well known, $x$ is a quasinilpotent (i.e., $x\in \mathcal{A}^{qnil}$) if and only if $1+\lambda x\in \mathcal{A}$ is invertible for any $\lambda\in {\Bbb C}$. Set $\mathcal{A}_w^{qnil}=\{x\in \mathcal{A}~\mid~ xw\in \mathcal{A}^{qnil}\}.$

Recently, the authors studied various generalized inverses with respect to a weight using a new limit-based approach in a Banach
algebra (see~\cite{CM1,CM2,CM3,CM4}). This inspires us to extend the study of generalized inverses to a broader context with a pair of weights.
The aim of this paper is to investigate when a Banach element can be expressed as the sum of a $[v,w]$-core invertible element and a quasinilpotent.
This leads to a new kind of generalized inverse, which is exactly a natural generalization of the generalized inverses mentioned above.

In Section 2, we introduce generalized weighted core inverse in terms of a new kind of generalized weighted core decomposition.
Many new properties of the $w$-core inverse and $w$-core-EP inverse are thereby obtained. We adopt:

\begin{defn} An element $a\in \mathcal{A}$ has generalized $[v,w]$-core decomposition if there exist $x,y\in \mathcal{A}$ such that $$a=x+y, ywx=0, x\in \mathcal{A}_{[v,w]}^{\tiny\textcircled{\#}}, y\in \mathcal{A}_w^{qnil}.$$\end{defn}

Let $a,w\in \mathcal{A}$ and $v\in \mathcal{A}^{-1}$. We prove that $a\in \mathcal{A}$ has generalized $[v,w]$-core decomposition if and only if there exists $x\in \mathcal{A}$ such that
$$x=awxvx, (vawx)^*=vawx, \lim\limits_{n\to \infty}||(aw)^n-xv(aw)^{n+1}||^{\frac{1}{n}}=0.$$

Recall that an element $a\in \mathcal{A}$ possesses the generalized
$w$-Drazin inverse $x$ if $x$ is the unique element in
$\mathcal{A}$ satisfying $awx=xwa, xwawx=x ~\mbox{and}~ a-awxwa\in \mathcal{A}^{qnil}.$
We denote this unique element by $a^{d,w}$ (see~\cite{M1}). Let $a\in \mathcal{A}^{d,w}$. In Section 3, we characterize the generalized $[v,w]$-core inverse by its polar-like and principal ideal properties. We prove that $a\in \mathcal{A}_{[v,w]}^{\tiny\textcircled{d}}$ if and only if there exists an idempotent $p\in \mathcal{A}$ such that $aw+p\in \mathcal{A}^{-1}, (v^2p)^*=v^2p, p(aw)=p(aw)p\in \mathcal{A}^{qnil}$, if and only if there exists $x\in \mathcal{A}$ such that $awxvx=x, x\mathcal{A}=a^{d,w}\mathcal{A}, \mathcal{A}x=\mathcal{A}(va^{d,w})^*.$

In Section 4, we establish the representation of the $[v,w]$-core-EP inverse through weighted generalized Drazin inverse. We prove that $a\in R_{[v,w]}^{\tiny\textcircled{d}}$ if and only if $a^{d,w}\in R_{[v,w]}^{\tiny\textcircled{\#}}$ and $a_{[v,w]}^{\tiny\textcircled{d}}=[a^{d,w}w]^2(a^{d,w})_{(v,w)}^{\tiny\textcircled{\#}}.$

Finally, Section 5 is devoted to applications in the complex setting. We derive new explicit expressions for the weighted core-EP inverse involving two weighted matrices.

Throughout the paper, $\mathcal{A}$ denotes the Banach *-algebras with an identity, $w\in \mathcal{A}$ and
$v\in \mathcal{A}^{-1}$. Here, $\mathcal{A}^{-1}$ denotes the set of all invertible elements in $\mathcal{A}$.
We use $\mathcal{A}^{d,w}$ and $\mathcal{A}_w^{\tiny\textcircled{\#}}$ to stand the sets of
all weighted g-Drazin and $w$-core invertible elements in $\mathcal{A}$, respectively.
Let ${\Bbb C}^{n\times n}$ be the Banach algebra of all $n\times n$ complex matrices with conjugate transpose $*$.

\section{generalized $[v,w]$-core decomposition}

The aim of this section is to introduce the notion of the generalized $[v,w]$-core inverse in a Banach *-algebra. We begin with

\begin{lem} Let $a,w\in \mathcal{A}$ and $v\in \mathcal{A}^{-1}$. Then the following are equivalent:\end{lem}
\begin{enumerate}
\item [(1)] $a\in \mathcal{A}_{[v,w]}^{\tiny\textcircled{\#}}$.
\vspace{-.5mm}
\item [(2)] There exists a unique $x\in \mathcal{A}$ such that
$$awxvx=x, xvawa=a, (vawx)^*=vawx.$$
\end{enumerate}
\begin{proof} $(1)\Rightarrow (2)$ Suppose that there exist $x,y\in \mathcal{A}$ such that
$$\begin{array}{c}
awxvx=x, xvawa=a, (vawx)^*=vawx;\\
awyvy=y, yvawa=a, (vawy)^*=vawy.
\end{array}$$ Then we have $$\begin{array}{rll}
awxva&=&awxv[xvawa]=[awxvx][vawa]=xvawa=a,\\
xvawx&=&xvaw[awxvx]=[xvawa][wxvx]=awxvx=x.
\end{array}$$ Likewise, we have
$awyva=a, yvawy=y$. Therefore $$\begin{array}{rll}
x&=&xvawx=xv[awyva]wx=xvaw[awyvy]vawx=[xvawa]wyvyvawx\\
&=&[awyvy]vawx=yvawx=(yvawy)vawx=y[vawy]vawx\\
&=&y[vawy]^*vawx=y(wy)^*a^*v^*vawx=y(wy)^*a^*v^*(vawx)^*\\
&=&=y(wy)^*(awxva)^*v^*=y(wy)^*a^*v^*\\
&=&y(vawy)^*=yvawy=y,
\end{array}$$ as required.

$(2)\Rightarrow (1)$ This is trivial.\end{proof}

\begin{thm} Let $a,w\in \mathcal{A}$ and $v\in \mathcal{A}^{-1}$. Then the following are equivalent:\end{thm}
\begin{enumerate}
\item [(1)] $a\in \mathcal{A}$ has generalized $[v,w]$-core decomposition.
\vspace{-.5mm}
\item [(2)] There exists a pair $x,y\in \mathcal{A}$ such that $$a=x+y, (vx)^*vy=0, ywx=0, x\in \mathcal{A}_{[v,w]}^{\tiny\textcircled{\#}}, y\in \mathcal{A}_w^{qnil}.$$
\vspace{-.5mm}
\item [(3)] There exists $x\in \mathcal{A}$ such that $$\begin{array}{c}
x=awxvx, (vawx)^*=vawx, \lim\limits_{n\to \infty}||(aw)^n-(xv)(aw)^{n+1}||^{\frac{1}{n}}=0.
\end{array}$$
\end{enumerate}
\begin{proof} $(1)\Rightarrow (3)$ By hypothesis, there exist $z,y\in \mathcal{A}$ such that $$a=z+y, ywz=0, z\in
\mathcal{A}_{[v,w]}^{\tiny\textcircled{\#}}, y\in \mathcal{A}_w^{qnil}.$$ Set $x=z_{[v,w]}^{\tiny\textcircled{\#}}$. Then $$\begin{array}{rll}
awx&=&(z+y)wz_{[v,w]}^{\tiny\textcircled{\#}}=zwz_{[v,w]}^{\tiny\textcircled{\#}},\\
vawx&=&vzwz_{[v,w]}^{\tiny\textcircled{\#}},\\
(vawx)^*&=&(vzwz_{[v,w]}^{\tiny\textcircled{\#}})^*=vzwz_{[v,w]}^{\tiny\textcircled{\#}}=vawx,\\
awxvx&=&(awx)(vx)=[zwz_{[v,w]}^{\tiny\textcircled{\#}}]vz_{[v,w]}^{\tiny\textcircled{\#}}=z_{[v,w]}^{\tiny\textcircled{\#}}=x,\\
xvawx&=&z_{[v,w]}^{\tiny\textcircled{\#}}vzwz_{[v,w]}^{\tiny\textcircled{\#}}=zwz_{[v,w]}^{\tiny\textcircled{\#}}=x.
\end{array}$$

As $ywz=0$, we see that $$\begin{array}{rl}
&[aw-xv(aw)^2]zw\\
=&[zw+yw-z^{e,\tiny\textcircled{\#}}v(zw+yw)^2]zw\\
=&[zw+yw-z^{e,\tiny\textcircled{\#}}v\big((zw)^2+zwyw+ywzw+(yw)^2\big)]zw\\
=&(zw)^2+(ywzw)-z^{e,\tiny\textcircled{\#}}v[(zw)^3+zw(ywz)w+(ywz)wzw+yw(ywz)w]\\
=&[z-z^{e,\tiny\textcircled{\#}}vzwz]wzw\\
=&[zwz-z^{e,\tiny\textcircled{\#}}v(zw)^2z]w=0.
\end{array}$$
Thus we have
$$\begin{array}{rll}
||(aw)^n-xv(aw)^{n+1}||^{\frac{1}{n}}&=&||[aw-xv(aw^2)](aw)^{n-1}||^{\frac{1}{n}}\\
&=&||[aw-xv(aw)^2][zw+yw]^{n-1}||^{\frac{1}{n}}\\
&=&||[aw-xv(aw)^2](yw)^{n-1}||^{\frac{1}{n}}\\
&\leq&||aw-xw(aw)^2||^{\frac{1}{n}}[||(yw)^{n-1}||^{\frac{1}{n-1}}]^{1-\frac{1}{n}}.
\end{array}$$
Since $y\in \mathcal{A}_w^{qnil}$, we deduce that $$\lim_{n\to \infty}||(aw)^n-xv(aw)^{n+1}||^{\frac{1}{n}}=0,$$ as required.

$(3)\Rightarrow (2)$ By hypotheses, there exists $x\in \mathcal{A}$ such that $$x=awxvx, (vawx)^*=vawx, \lim\limits_{n\to \infty}||(aw)^n-xv(aw)^{n+1}||^{\frac{1}{n}}=0.$$ Then $xvawx=x$ and
$[awxv(aw)^2]x=(aw)^2x.$ Moreover, we have $$vaw(xvaw)x=v(aw)x.$$ Set $z=awxva$ and $y=a-awxva$. We verify that
$$\begin{array}{rll}
ywz&=&(a-awxva)wawxva=(aw)^2xva-[awxv(aw)^2]xva\\
&=&(aw)^2xva-(aw)^2xva=0,\\
(vz)^*vy&=&(vawxva)^*v(a-awxva)\\
&=&(va)^*vawxv(a-awxva)=(va)^*[vawxva-vaw(xvaw)xva]\\
&=&(va)^*[vawxva-vawxva]=0.
\end{array}$$

We claim that $z\in \mathcal{A}_{[v,w]}^{\tiny\textcircled{\#}}$ and $z_{[v,w]}^{\tiny\textcircled{\#}}=x$.

Claim 1. $zwxvx=x$.

$$\begin{array}{rll}
zwxvx&=&awxv(awxvx) \\
&=&awxvx=x.
\end{array}$$

Claim 2. $xvzwx=x.$

$$\begin{array}{rll}
xvzwx&=&xvaw(xvawx)\\
&=&xvawx=x.
\end{array}$$

Claim 3. $xvzwz=z.$
Since $zwz\in x\mathcal{A}$, we deduce that
$$\begin{array}{rll}
xvzwz&=&(xvawx)vawawxva\\
&=&[xv(aw)^2]xva=awxva=z.
\end{array}$$

Claim 5. $(vzwx)^*=vzwx.$

$$\begin{array}{rll}
vzwx&=&vaw(xvawx)=vawx,\\
(vzwx)^*&=&(vawx)^*=vawx=vzwx.
\end{array}$$

Therefore $z\in \mathcal{A}_{[v,w]}^{\tiny\textcircled{\#}}$ and $z_{[v,w]}^{\tiny\textcircled{\#}}=x$. Moreover, we see that
$$\begin{array}{rll}
||(aw)^n-awxv(aw)^{n}||&=&||[aw-awxvaw](aw)^{n-1}||\\
&=&||yw(aw)^{n-1}||=||ywaw(aw)^{n-2}||\\
&=&||yw(z+y)w(aw)^{n-2}||=||(yw)^2(aw)^{n-2}||\\
&=&\cdots =||(yw)^n||.
\end{array}$$ Therefore $$\lim\limits_{n\to \infty}||(yw)^n||^{\frac{1}{n}}=0,$$
and then $y\in \mathcal{A}_w^{qnil}$, as required.

$(2)\Rightarrow (1)$ This is trivial.\end{proof}

\begin{cor} Let $a,w\in \mathcal{A}$ and $v\in \mathcal{A}^{-1}$. Then the following are equivalent:\end{cor}
\begin{enumerate}
\item [(1)] $a\in \mathcal{A}$ has generalized $[v,w]$-core decomposition.
\vspace{-.5mm}
\item [(2)] There exists a unique pair $x,y\in \mathcal{A}$ such that $$a=x+y, (vx)^*vy=0, ywx=0, x\in \mathcal{A}_{[v,w]}^{\tiny\textcircled{\#}}, y\in \mathcal{A}_w^{qnil}.$$
\end{enumerate}
\begin{proof} $(1)\Rightarrow (2)$ In view of Theorem 2.2, there exists a pair $x,y\in \mathcal{A}$ such that $$a=x+y, (vx)^*vy=ywx=0, x\in
\mathcal{A}_{[v,w]}^{\tiny\textcircled{\#}}, y\in \mathcal{A}_w^{qnil}.$$ Suppose that
there exist $b,c\in \mathcal{A}$ such that $$a=b+c, (vb)^*vc=cwb=0, b\in
\mathcal{A}_{[v,w]}^{\tiny\textcircled{\#}}, c\in \mathcal{A}_w^{qnil}.$$

Claim 1. Set $z=x_{[v,w]}^{\tiny\textcircled{\#}}$ and $d=b_{[v,w]}^{\tiny\textcircled{\#}}$. As in the proof of Theorem 2.2, we have
$$\begin{array}{c}
z=awzvz, (vawz)^*=vawz, zvawz=z,\\
\lim\limits_{n\to \infty}||(aw)^n-(zv)(aw)^{n+1}||^{\frac{1}{n}}=0;\\
d=awdvd, (vawd)^*=vawd, dvawd=d,\\
\lim\limits_{n\to \infty}||(aw)^n-(dv)(aw)^{n+1}||^{\frac{1}{n}}=0.
\end{array}$$

Claim 2. $awz=awd$. For any $n\in {\Bbb N}$, we have $z=awzvz=(aw)^2z(vz)^2=\cdots =(aw)^{n-1}z(vz)^{n-1}$.
Then we get $$\begin{array}{rll}
awz-awzvawz&=&(aw)^nz(vz)^{n-1}-awzv(aw)^{n}z(vz)^{n-1}\\
&=&[(aw)^n-awzv(aw)^n]z(vz)^{n-1}.
\end{array}$$
Hence $\lim\limits_{n\to \infty}||awz-awzvawz||^{\frac{1}{n}}=0.$ Then $awz=awzvawz.$ Likewise,
we have $awz=awdvawz.$ Moreover, we verify that $aw(dvawd)vawz=awdv(awdvaw)z=awdvawz$, and so
$awdvawz=aw(dvawd)vawz$ Thus $awd=awzvawd.$ Similarly, we have $(awdvaw)d=awd.$

We observe that
$$\begin{array}{rll}
awz&=&aw(zvawz)=awzv(aw)z=awzv(awdva)wz\\
&=&awzvaw(awdvd)vawz=(awzvawaw)dvdvawz\\
&=&aw(awdvd)(vawz)=awdvawz=aw(dvawd)vawz\\
&=&awd(vawd)^*vawz=awd(vawzvawd)^*vawz\\
&=&awd(d^*w^*a^*v^*)(vawz)^*vawz=awdd^*w^*a^*v^*(vawz)vawz\\
&=&awdd^*w^*a^*v^*(vawzvawz)=awdd^*w^*a^*v^*vawz\\
&=&awd(d^*w^*a^*v^*)(vawz)^*=awd(vawzvawd)^*\\
&=&awd(vawd)^*=awd(vawd)=awdvawd=(awdvaw)d=awd.
\end{array}$$

Claim 3. $x=b$. We prove that $$\begin{array}{rll}
vawzvy&=&v(x+y)wzvy=vxwzvy=(vxwz)^*vy=(wz)^*\big((vx)^*vy\big)=0,\\
vawzva&=&vawzv(x+y)=vawzvx=v(x+y)wzvx=vxwzvx=vx.\end{array}$$ Hence,
$vx=vawzva$. Likewise, $vb=vawdva$. By Claim 2, we see that
$vx=v(awz)(va)=v(awd)(va)=vb$. Since $v\in \mathcal{A}^{-1}$, we see that $x=b$. Therefore $c=a-b=a-x=y$, as required.

$(2)\Rightarrow (1)$ This is trivial.\end{proof}

\begin{thm} Let $a,w\in \mathcal{A}$ and $v\in \mathcal{A}^{-1}$. Then the following are equivalent:\end{thm}
\begin{enumerate}
\item [(1)] $a\in \mathcal{A}$ has generalized $[v,w]$-core decomposition.
\vspace{-.5mm}
\item [(2)] There exists unique $x\in \mathcal{A}$ such that $$\begin{array}{c}
x=awxvx, (vawx)^*=vawx, \lim\limits_{n\to \infty}||(aw)^n-xv(aw)^{n+1}||^{\frac{1}{n}}=0.
\end{array}$$
\end{enumerate}
\begin{proof} $(2)\Rightarrow (1)$ This is obvious by Theorem 2.2.

$(1)\Rightarrow (2)$ By hypothesis, there exist $x,y\in \mathcal{A}$ such that $$\begin{array}{c}
x=awxvx, (vawx)^*=vawx, \lim\limits_{n\to \infty}||(aw)^n-(xv)(aw)^{n+1}||^{\frac{1}{n}}=0,\\
y=awyvy, (vawy)^*=vawy, \lim\limits_{n\to \infty}||(aw)^n-(yv)(aw)^{n+1}||^{\frac{1}{n}}=0.
\end{array}$$

Set $z=(aw)(xv)a, y=a-(aw)(xv)a; b=(aw)(yv)a$ and $c=a-(aw)(yv)a$. As in the proof of Theorem 2.2, we have generalized $[v,w]$-core decomposition
$$\begin{array}{c}
a=z+y, (vz)^*vy=ywz=0, z\in \mathcal{A}_{[v,w]}^{\tiny\textcircled{\#}}, y\in \mathcal{A}_w^{qnil};\\
a=b+c, (vb)^*vc=cwb=0, b\in \mathcal{A}_{[v,w]}^{\tiny\textcircled{\#}}, c\in \mathcal{A}_w^{qnil}.
\end{array}$$

In view of Corollary 2.3, we have $z=b$.
As in the proof of Theorem 2.2, we get
$$x=z_{[v,w]}^{\tiny\textcircled{\#}}=b_{[v,w]}^{\tiny\textcircled{\#}}=y.$$ Accordingly, $x=y$, the result follows.\end{proof}

We denote the unique $x$ described above by $a_{[v,w]}^{\tiny\textcircled{d}}$, and call it the generalized $[v,w]$-core inverse of $a$.

\begin{cor} Let $a=x+y$ be the generalized $[v,w]$-core decomposition of $a\in \mathcal{A}$. Then $a_{[v,w]}^{\tiny\textcircled{d}}=x_{[v,w]}^{\tiny\textcircled{\#}}$.\end{cor}
\begin{proof} Let $a=x+y$ be the generalized $[v,w]$-core decomposition of $a\in \mathcal{A}$. Similarly to the proof of Theorem 2.2,
$x_{[v,w]}^{\tiny\textcircled{\#}}$ is the generalized $[v,w]$-core inverse of $a$. So the theorem is true.\end{proof}

\begin{thm} Let $a\in \mathcal{A}_{[v,w]}^{\tiny\textcircled{d}}$ and $b\in \mathcal{A}_w^{qnil}$. If $(va)^*vb=0$ and $bwa=0$, then $a+b\in \mathcal{A}_{[v,w]}^{\tiny\textcircled{d}}$. In this case, $$(a+b)_{[v,w]}^{\tiny\textcircled{d}}=a_{[v,w]}^{\tiny\textcircled{d}}.$$\end{thm}
\begin{proof} Since $a\in \mathcal{A}_{[v,w]}^{\tiny\textcircled{d}}$, it follows by Theorem 2.2 that there exist
$x\in \mathcal{A}_{[v,w]}^{\tiny\textcircled{\#}}$ and $y\in \mathcal{A}_w^{qnil}$ such that $a=x+y, (vx)^*vy=0, ywx=0$. As in the proof of Theorem 2.2,
$x=(aw)(a_{[v,w]}^{\tiny\textcircled{d}}v)a$ and $y=a-(aw)(a_{[v,w]}^{\tiny\textcircled{d}}v)a$.
Then $a=x+(y+b)$. Since $bwy=(bwa)[1-wa_{[v,w]}^{\tiny\textcircled{d}}va]=0$, it follows by ~\cite[Lemma2.4]{CK} that $y+b\in \mathcal{A}_w^{qnil}$.
Obviously, $(vx)^*v(y+b)=(vx)^*vy+[wa_{[v,w]}^{\tiny\textcircled{d}}va]^*[(va)^*vb]=0$. By virtue of Theorem 2.2,
$a+b\in \mathcal{A}_{[v,w]}^{\tiny\textcircled{d}}$. In this case, $$(a+b)_{[v,w]}^{\tiny\textcircled{d}}=x_{[v,w]}^{\tiny\textcircled{\#}}=a_{[v,w]}^{\tiny\textcircled{d}},$$ as desired.\end{proof}

\begin{cor} Let $M=\left(
\begin{array}{cc}
a&b\\
0&d
\end{array}
\right)$ with $a,d\in \mathcal{A}_{[v,w]}^{\tiny\textcircled{d}}$. If $(va)^*vb=0$ and $bwd=0$, then
$M\in M_2(\mathcal{A})_{[vI_2,wI_2]}^{\tiny\textcircled{d}}$.\end{cor}
\begin{proof} Write $M=P+Q$, where
$P=\left(
\begin{array}{cc}
a&0\\
0&d
\end{array}
\right), Q=\left(
\begin{array}{cc}
0&b\\
0&0
\end{array}
\right).$ We directly verify that $P\in M_2(\mathcal{A})_{[vI_2,wI_2]}^{\tiny\textcircled{d}}$ and $P_{[vI_2,wI_2]}^{\tiny\textcircled{d}}=\left(
\begin{array}{cc}
a_{[v,w]}^{\tiny\textcircled{d}}&0\\
0&d_{[v,w]}^{\tiny\textcircled{d}}
\end{array}
\right)$. Obviously, $[Q(wI_2)]^2=0$. By hypothesis, we verify that
$$Q(wI_2)P=0, \big((vI_2)P\big)^*(vI_2)Q=0.$$ By virtue of Theorem 2.6,
$M\in M_2(\mathcal{A})_{[vI_2,wI_2]}^{\tiny\textcircled{d}}$, as asserted.\end{proof}

\section{weighted polar-like property and principal ideals}

The aim of this section is to characterize the generalized $[v,w]$-core inverse by using a kind of polar-like property. We then
establish equivalent conditions involving principal ideals under which an element in a Banach algebra has a generalized $[v,w]$-core inverse.

Let $a\in \mathcal{A}$ and let $v\in \mathcal{A}^{-1}$ be Hermitian (i.e., $v$ is invertible and $v^*=v$).
An element $a\in \mathcal{A}$ has $v$-core inverse if there exists $x\in \mathcal{A}$ such that $$ax^2=x, (vax)^*=vax, xa^2=a.$$ If such $x$ exists,
it is unique, and denote it by $a_v^{\tiny\textcircled{\#}}$ (see~\cite{MCD}).

\begin{lem} Let $a\in \mathcal{A}^{d,w}$ and $v\in \mathcal{A}^{-1}$ be Hermitian. Then the following are equivalent:\end{lem}
\begin{enumerate}
\item [(1)] $a\in \mathcal{A}_{[v,w]}^{\tiny\textcircled{\#}}$.
\vspace{-.5mm}
\item [(2)] $aw\mathcal{A}=a\mathcal{A}$ and $aw\in \mathcal{A}_{v^2}^{\tiny\textcircled{\#}}$.
\end{enumerate}
In this case, $a_{[v,w]}^{\tiny\textcircled{\#}}=(aw)_{v^2}^{\tiny\textcircled{\#}}v^{-1}.$
\begin{proof} $\Rightarrow (2)$ Set $x=a_{[v,w]}^{\tiny\textcircled{\#}}$. Then
$$awxvx=x, xvawa=a, (vawx)^*=vawx.$$ As in the proof in Lemma 2.1, we have $awxva=a$. Then
$aw\mathcal{A}=a\mathcal{A}$. Since $(vawx)^*=vawx,$ we directly verify that
$$(v^2awxv)^*=v^*(vawx)^*v^*=v^*(vawx)v^*=v^2(awxv).$$ Since $aw(xv)^2=xv$ and $(xv)(aw)^2=aw$.
Therefore $aw\in \mathcal{A}_{v^2}^{\tiny\textcircled{\#}}$, as required.

$(2)\Rightarrow (1)$ Set $x=(aw)_{v^2}^{\tiny\textcircled{\#}}v^{-1}$. Since $aw\mathcal{A}=a\mathcal{A}$. Then we check that
$$\begin{array}{rll}
awxvx&=&aw[(aw)_{v^2}^{\tiny\textcircled{\#}}]^2v^{-1}=(aw)_{v^2}^{\tiny\textcircled{\#}}v^{-1}=x,\\
xvawa&=&(aw)_{v^2}^{\tiny\textcircled{\#}}awa=a,\\
vawx&=&vaw(aw)_{v^2}^{\tiny\textcircled{\#}}v^{-1}=v^{-1}[v^2aw(aw)_{v^2}^{\tiny\textcircled{\#}}]v^{-1},\\    .
(vawx)^*&=&vawx.
\end{array}$$ Hence, $a_{[v,w]}^{\tiny\textcircled{\#}}=x$. Therefore
$$a_{[v,w]}^{\tiny\textcircled{\#}}=(aw)_{v^2}^{\tiny\textcircled{\#}}v^{-1},$$ as desired.\end{proof}

We are ready to prove:

\begin{thm} Let $a\in \mathcal{A}^{d,w}$ and $v\in \mathcal{A}^{-1}$ be Hermitian. Then $a\in \mathcal{A}_{[v,w]}^{\tiny\textcircled{d}}$
if and only if there exists an idempotent $p\in \mathcal{A}$ such that $$aw+p\in \mathcal{A}^{-1}, (v^2p)^*=v^2p, p(aw)=p(aw)p\in \mathcal{A}^{qnil}.$$
\end{thm}
\begin{proof} $\Longrightarrow $ Since $a\in \mathcal{A}_{[v,w]}^{\tiny\textcircled{d}}$, by using Theorem 2.2, there exist $x,y\in \mathcal{A}$ such that $$a=x+y, (vx)^*vy=ywx=0, x\in
\mathcal{A}_{[v,w]}^{\tiny\textcircled{\#}}, y\in \mathcal{A}_w^{qnil}.$$
Set $y=x_{[v,w]}^{\tiny\textcircled{\#}}$. By virtue of Lemma 3.1, $xw\in \mathcal{A}_{v^2}^{\tiny\textcircled{\#}}$. Set
  $z=(xw)_{v^2}^{\tiny\textcircled{\#}}$. Then we have
$$(xw)z^2=z, z(xw)^2=xw, (v^2xwz)^*=v^2xwz.$$
Let $p=1-awz$. Then $p=1-(x+y)wz=1-xwz$. Clearly, $p^2=1-xwz-xwz+xw[z(xw)^2]z^2=1-xwz-xwz+(xw)^2z^2=1-xwz=p$. We check that
$$\begin{array}{rll}
(xw+p)^{-1}&=&z+1-zxw,\\
aw+p&=&(x+y)w+p=(xw+p)+yw\\
&=&(xw+p)[1+(xw+p)^{-1}yw],\\
yw(xw+p)^{-1}&=&yw[z+1-zxw]=yw\in \mathcal{A}^{qnil},\\
(xw+p)^{-1}yw&\in &\mathcal{A}^{qnil}.
\end{array}$$ Therefore $aw+p\in \mathcal{A}^{-1}$.

Since $(v^2xwz)^*=v^2xwz$, we see that $(v^2p)^*=v^2p$. We verify that
$$\begin{array}{rll}
paw&=&p(x+y)w=(1-xwz)yw\\
&=&yw-v^{-1}(vxwz)yw\\
&=&yw-v^{-1}(vxwz)^*yw\\
&=&yw-v^{-1}(wz)^*[(vx)^*vy]w\\
&=&yw\in \mathcal{A}^{qnil},
\end{array}$$ and so $paw\in \mathcal{A}^{qnil}$. Moreover, we have
$paw(1-p)=ywawz=yw(x+y)wx_{v^2}^{\tiny\textcircled{\#}}=(ywx)wx_{v^2}^{\tiny\textcircled{\#}}=0$, as required.

$\Longleftarrow $ By hypothesis, there exists an idempotent $p\in \mathcal{A}$ such that
$$u:=aw+p\in \mathcal{A}^{-1}, (v^2p)^*=v^2p, p(aw)=p(aw)p\in \mathcal{A}^{qnil}.$$
Step 1. Set $s=(1-p)(aw)$. We shall prove that $s\in \mathcal{A}_{v^2}^{\tiny\textcircled{\#}}$.

Claim 1. $s\in \mathcal{A}^d$.

Since $paw(1-p)=0$, we have $a=pawp+(1-p)awp+(1-p)aw(1-p)$, i.e.,
$aw=\left(
\begin{array}{cc}
pawp&0\\
(1-p)awp&(1-p)aw(1-p)
\end{array}
\right)_p$. By hypothesis, $paw\in \mathcal{A}^{qnil}$, then $pawp\in \mathcal{A}^{qnil}$ by using Cline's formula (see~\cite{CM}). This implies that
$pawp\in \mathcal{A}^d$. Since $aw\in \mathcal{A}^d$, it follows by ~\cite[Theorem 2.3]{CK} that
$(1-p)aw(1-p)\in \big((1-p)\mathcal{A}(1-p)\big)^d\subseteq \mathcal{A}^d.$ By Cline's formula again, $s=(1-p)aw\in \mathcal{A}^d,$ as desired.

Claim 2. $s\in \mathcal{A}^{\#}$. Observing that $$\begin{array}{rll}
u^{-1}s^2&=&u^{-1}(1-p)aw(1-p)aw=u^{-1}aw(1-p)aw-u^{-1}[paw(1-p)]aw\\
&=&u^{-1}aw(1-p)aw=(aw+p)^{-1}(aw+p)(1-p)aw=(1-p)aw=s,
\end{array}$$ we have $\mathcal{A}s=\mathcal{A}s^2$. In light of ~\cite[Lemma 2.1]{CM4},
$s\in \mathcal{A}^{\#}$.

Claim 3. $s\in \mathcal{A}_{v^2}^{\tiny\textcircled{\#}}$.

Obviously, $su^{-1}=(1-p)aw(aw+p)^{-1}=(1-p)(aw+p)(aw+p)^{-1}=1-p$. Then $su^{-1}s=(1-p)s=s$ and $(v^2su^{-1})^*=(v^2-v^2p)^*=(v^2)^*-(v^2p)^*
=v^2-v^2p=v^2su^{-1}$. According to ~\cite[Theorem 2.1]{MCD}, $s\in \mathcal{A}_{v^2}^{\tiny\textcircled{\#}}$.

Step 2. Let $q=1-ss_{v^2}^{\tiny\textcircled{\#}}$.
Then $q^2=q, (v^2q)^*=v^2q$ and $qs=0$. We verify that
$$(s+1-ss_{v^2}^{\tiny\textcircled{\#}})(s_{v^2}^{\tiny\textcircled{\#}}+1-s_{v^2}^{\tiny\textcircled{\#}}s)=1
=(s_{v^2}^{\tiny\textcircled{\#}}+1-s_{v^2}^{\tiny\textcircled{\#}}s)
(s+1-ss_{v^2}^{\tiny\textcircled{\#}}).$$
Hence, $$(s+q)^{-1}=s_{v^2}^{\tiny\textcircled{\#}}+1-s_{v^2}^{\tiny\textcircled{\#}}s.$$

Set $t=p(aw)$. Then $aw=s+t$. We check that $$\begin{array}{rll}
s^*v^2t&=&[(1-p)aw]^*v^2(paw)=[(aw)^*(1-p)^*](v^2p)^*aw\\
&=&(aw)^*[(1-p)^*p^*]v^2aw=0,\\
ts&=&paw(1-p)aw=[paw(1-p)]aw=0,\\
t&=&paw\in \mathcal{A}^{qnil}.
\end{array}$$

Since $t(s+q)=t(s+1-ss_{v^2}^{\tiny\textcircled{\#}})=t$, we see that $t(s+q)^{-1}=t\in \mathcal{A}^{qnil}$.
By using Cline's formula (see~\cite{CM}), $(s+q)^{-1}t\in \mathcal{A}^{qnil}$. Hence, $1+(s+q)^{-1}t\in \mathcal{A}^{-1}$.
Therefore, we check that
$$\begin{array}{rll}
qaw&=&q(s+t)=qt=(1-ss_{v^2}^{\tiny\textcircled{\#}})t=v^{-2}(v^2-v^2ss_{v^2}^{\tiny\textcircled{\#}})t\\
&=&v^{-2}(v^2-v^2ss_{v}^{\tiny\textcircled{\#}})^*t=v^{-2}(v^2-(s_{v^2}^{\tiny\textcircled{\#}})^*s^*v^2)t\\
&=&t-v^{-2}(s_{v^2}^{\tiny\textcircled{\#}})^*(s^*v^2t)=t\in \mathcal{A}^{qnil},\\
qaw(1-q)&=&tss_{v^2}^{\tiny\textcircled{\#}}=0, qaw=qawq,\\
aw+q&=&(s+t)+q=(s+q)[1+(s+q)^{-1}t]\in \mathcal{A}^{-1}.
\end{array}$$

As $aw=s+t$ and $ts=0$, we verify that $$\begin{array}{rll}
1-q&=&ss_{v^2}^{\tiny\textcircled{\#}}=s^2(s_{v^2}^{\tiny\textcircled{\#}})^2\\
&=&(s+t)s(s_{v^2}^{\tiny\textcircled{\#}})^2\\
&=&(aw)[ss_{v^2}^{\tiny\textcircled{\#}}]s_{v^2}^{\tiny\textcircled{\#}}\\
&=&(aw)(1-q)s_{v^2}^{\tiny\textcircled{\#}}\\
&\in &(aw)(1-q)\mathcal{A}.
\end{array}$$

Accordingly, we conclude that $$aw+q\in \mathcal{A}^{-1}, (v^2q)^*=vq, 1-q\in (aw)(1-q)\mathcal{A}, q(aw)=q(aw)q\in \mathcal{A}^{qnil}.$$

Step 3. Set $x=(1-q)a$ and $y=qa$. Then $a=x+y, ywx=qaw[(1-q)a]=[qaw(1-q)]a=0$ and $y\in\mathcal{A}_w^{qnil}$.

Set $z=(aw+q)^{-1}(1-q)$. Write $1-q=aw(1-q)r$ for some $r\in \mathcal{A}$.
Then we verify that
$$\begin{array}{rll}
xwz&=&(1-q)aw(aw+q)^{-1}(1-q)\\
&=&(1-q)(aw+q)(aw+q)^{-1}(1-q)=1-q,\\
(xw)z^2&=&(1-q)z=(1-q)(aw+q)^{-1}(1-q)\\
&=&(1-q)(aw+q)^{-1}aw(1-q)r\\
&=&(1-q)(aw+q)^{-1}(aw+q)(1-q)r\\
&=&(1-q)r=(aw+q)^{-1}(1-q)=z,\\
(v^2xwz)^*&=&[v^2(1-q)]^*=v^2(1-q)=v^2xwz,\\
z(xw)^2&=&(aw+q)^{-1}(1-q)aw(1-q)aw=(aw+q)^{-1}aw(1-q)aw\\
&=&(aw+q)^{-1}(aw+q)(1-q)aw=(1-q)aw=xw.
\end{array}$$
Thus $xw\in \mathcal{A}_{v^2}^{\tiny\textcircled{\#}}$ and $(xw)_{v^2}^{\tiny\textcircled{\#}}=z$.

Obviously, $(1-q)aw\mathcal{A}\subseteq (1-q)a\mathcal{A}$. Since $1-q=(1-q)^2=(1-q)aw(1-q)r$, we have
$(1-q)a\mathcal{A}\subseteq (1-q)aw\mathcal{A}$. Hence
$x\mathcal{A}=xw\mathcal{A}$. Accordingly, $x\in \mathcal{A}_{[v,w]}^{\tiny\textcircled{\#}}$ by Lemma 3.1. In light of
Theorem 2.2, $a\in \mathcal{A}_{[v,w]}^{\tiny\textcircled{d}}$, desired.\end{proof}

Set $$comm^*(a)=\{ x\in \mathcal{A}~\mid~(ax^2)a=(xax)a, x(a^2x)=(axa)x\}$$ (the weak commutant of $a$).

\begin{cor} Let $a\in \mathcal{A}^{d,w}$ and $v\in \mathcal{A}^{-1}$ be Hermitian. Then the following are equivalent:\end{cor}
\begin{enumerate}
\item [(1)] $a\in \mathcal{A}_{[v,w]}^{\tiny\textcircled{d}}$.
\vspace{-.5mm}
\item [(2)] There exists $x\in comm^*(aw)$ such that $x=awx^2, (v^2awx)^*=v^2awx, a-awxa\in \mathcal{A}_w^{qnil}.$
\end{enumerate}
\begin{proof} This is proved by Theorem 3.2 and ~\cite[Theorem 2.5]{CM4}.\end{proof}

\begin{cor} Let $a\in \mathcal{A}^{d,w}$. Then $a\in \mathcal{A}_w^{\tiny\textcircled{d}}$ if and only if there exists a projection $p\in \mathcal{A}$ such that $$aw+p\in \mathcal{A}^{-1}, p(aw)=p(aw)p\in \mathcal{A}^{qnil}.$$
\end{cor}
\begin{proof} This is obvious by choosing $v=1$ in Theorem 3.2.\end{proof}

We are ready to prove:

\begin{thm} Let $a\in \mathcal{A}^{d,w}$ and $v\in \mathcal{A}^{-1}$. Then the following are equivalent:\end{thm}
\begin{enumerate}
\item [(1)] $a\in \mathcal{A}_{[v,w]}^{\tiny\textcircled{d}}$.
\item [(2)] There exists $x\in \mathcal{A}$ such that $$awxvx=x, x\mathcal{A}=a^{d,w}\mathcal{A},
\mathcal{A}x=\mathcal{A}(va^{d,w})^*.$$
\end{enumerate} In this case, $a_{[v,w]}^{\tiny\textcircled{d}}=x.$
\begin{proof} $(1)\Rightarrow (2)$ Set $x=a_{[v,w]}^{\tiny\textcircled{d}}$. By virtue of Theorem 2.2, we have $$\begin{array}{c}
x=awxvx, (vawx)^*=vawx, \lim\limits_{n\to \infty}||(aw)^n-(xv)(aw)^{n+1}||^{\frac{1}{n}}=0.
\end{array}$$ Then $x=xvawx$. Since $aw\in \mathcal{A}^d$, we see that $$x=(aw)^{n}x(vx)^n=(aw)^d(aw)^{n+1}x(vx)^n\subseteq a^{d,w}\mathcal{A}.$$
On the other hand, we have $$\begin{array}{rll}
(aw)^d&=&(aw)^n((aw)^d)^{n+1}=(xv)(aw)^{n+1}((aw)^d)^{n+1}\\
&=&xv(aw)(aw)^d\in x\mathcal{A}.
\end{array}$$ Hence $x\mathcal{A}=a^{d,w}\mathcal{A}$. Since $xvawx=x, (vawx)^*=vawx$, we derive that
$x^*=(vawx)^*x^*=vawxx^*\in va^{d,w}\mathcal{A}$; hence, $\mathcal{A}x\subseteq \mathcal{A}(va^{d,w})^*$.
Since $$\begin{array}{rll}
a^{d,w}&=&(aw)[(aw)^d]^3a=(aw)(aw)^n[(aw)^d]^{n+1}[(aw)^d]^2a\\
&=&(aw)xv(aw)^{n+1}[(aw)^d]^{n+1}[(aw)^d]^2a=awxv[(aw)^d]^2a,
\end{array}$$ we have $va^{d,w}\in vawxv[(aw)^d]^2a=(vawx)^*v[(aw)^d]^2a$; hence,
$(va^{d,w})^*=[v((aw)^d)^2a]^*vawx\in \mathcal{A}x$. Therefore $\mathcal{A}x=\mathcal{A}(va^{d,w})^*,$ required.

$(2)\Rightarrow (1)$ Obviously, $x^*\mathcal{A}=va^{d,w}\mathcal{A}=vx\mathcal{A}$.
Since $x=awxvx$, we have $vx=vawxvx$, and so $(1-vawx)vx=0$. This implies that $(1-vawx)x^*=0$.
That is, $x^*=vawxx^*$. This implies that $(vawx)^*=x^*w^*a^*v^*=(vawx)(vawx)^*$; hence,
$vawx=[(vawx)(vawx)^*]^*=(vawx)(vawx)^*=(vawx)^*$. Furthermore,
$x=x(vawx)^*=xvawx$.

Write $(aw)^d=xt$ for some $t\in \mathcal{A}$. For any $n\in {\Bbb N}$, we have
$$\begin{array}{rll}
(aw)^d(aw)^{n+1}-(xv)(aw)^d(aw)^{n+2}&=&[(aw)^d-(xv)aw(aw)^d](aw)^{n+1}\\
&=&[x-xvawx]t(aw)^{n+1}=0.
\end{array}$$ Then
$$\begin{array}{rl}
&||(aw)^n-(xv)(aw)^{n+1}||\\
=&||(aw)^n-(aw)^d(aw)^{n+1}-(xv)[(aw)^n-(aw)^d(aw)^{n+1}](aw)\\
+&[(aw)^d(aw)^{n+1}-(xv)(aw)^d(aw)^{n+2}|| \\
=&||(aw)^n-(aw)^d(aw)^{n+1}-(xv)[(aw)^n-(aw)^d(aw)^{n+1}](aw)||\\
\leq&[1+||xv||||aw||]||(aw)^n-(aw)^d(aw)^{n+1}||.
\end{array}$$
Hence, $||(aw)^n-(xv)(aw)^{n+1}||^{\frac{1}{n}}\leq [1+||xv||||aw||]^{\frac{1}{n}}||(aw)^n-(aw)^d(aw)^{n+1}||^{\frac{1}{n}}.$
Therefore $\lim\limits_{n\to \infty}||(aw)^n-(xv)(aw)^{n+1}||^{\frac{1}{n}}=0$, as desired.\end{proof}

\begin{cor} Let $a\in \mathcal{A}^{d,w}$ and $v\in \mathcal{A}^{-1}$ be Hermitian. Then the following are equivalent:\end{cor}
\begin{enumerate}
\item [(1)] $a\in \mathcal{A}_{[v,w]}^{\tiny\textcircled{d}}$.
\item [(2)] There exists $x\in \mathcal{A}$ such that $$xvawx=x, (vawx)^*=vawx, x\mathcal{A}=a^{d,w}\mathcal{A},
\mathcal{A}x=\mathcal{A}(va^{d,w})^*.$$
\end{enumerate} In this case, $a_{[v,w]}^{\tiny\textcircled{d}}=x.$
\begin{proof} $(1)\Rightarrow (2)$ This is obvious by Theorem 3.5.

$(2)\Rightarrow (1)$ Since $xvawx=x$ and $(vawx)^*=vawx$, we have $(vawx)^*x^*=x^*$.
Since $\mathcal{A}x=\mathcal{A}(va^{d,w})^*$, we see that $(va^{d,w})^*=zx$ for some $z\in \mathcal{A}$. Then
$v(aw)^d=x^*z^*$. This implies that $vawxv(aw)^d=v(aw)^d$. As $x\mathcal{A}=a^{d,w}\mathcal{A}$, we can find some $s\in \mathcal{A}$ such that
$x=(aw)^ds$. We infer that $vawxv[(aw)^ds]=v(aw)^ds$
Thus $vawxvx=vx$, and so $awxvx=x$. We obtain the result by Theorem 3.5.\end{proof}

Let $a,b,c\in \mathcal{A}$. An element $a$ has $(b,c)$-inverse provide that there exists $x\in \mathcal{A}$ such that $$xab=b, cax=c
~\mbox{and}~ x\in b\mathcal{A}x\bigcap x\mathcal{A}c.$$ If such $x$ exists, it is unique and denote it by $a^{(b,c)}$ (see~\cite{DM}).

\begin{thm} Let $a\in \mathcal{A}^{d,w}$ and $v\in \mathcal{A}^{-1}$. Then the following are equivalent:\end{thm}
\begin{enumerate}
\item [(1)] $a\in \mathcal{A}_{[v,w]}^{\tiny\textcircled{d}}$.
\item [(2)] $vaw\in \mathcal{A}^{\big(a^{d,w}, (va^{d,w})^*\big)}$ and $vaw(vaw)^{\big(a^{d,w}, (va^{d,w})^*\big)}$ is Hermitian.
\end{enumerate} In this case, $a_{[v,w]}^{\tiny\textcircled{d}}=(vaw)^{\big(a^{d,w}, (va^{d,w})^*\big)}.$
\begin{proof} $(1)\Rightarrow (2)$ Set $x=a_{[v,w]}^{\tiny\textcircled{d}}$. Then $x(vaw)x=x$ by Corollary 3.6.

Claim 1. $x(vaw)a^{d,w}=a^{d,w},(va^{d,w})^*(vaw)x=(va^{d,w})^*.$

$$\begin{array}{rll}
x(vaw)(aw)^da&=&x(vaw)(aw)^da\\
&=&(aw)^d[(aw)(xv)(aw)](aw)^da\\
&=&(aw)^d(aw)(aw)^da\\
&=&(aw)^da,\\
\end{array}$$ Hence, $x(vaw)a^{d,w}=a^{d,w}$.

$$\begin{array}{rll}
(v(aw)^da)^*(vaw)x&=&(v(aw)^da)^*(vawx)^*\\
&=&[vawxv(aw)^da]^*=[v(aw)^da]^*,\\
\end{array}$$ Since $a^{d,w}=[(aw)^d]^2a$, we have $(va^{d,w})^*(vaw)x=(va^{d,w})^*.$

Claim 2. $x\in a^{d,w}\mathcal{A}x.$

$$x=awxvx=(aw)^d(aw)^2xvx\in (aw)^da\mathcal{A}x.$$

Claim 3. $x\in x\mathcal{A}(va^{d,w})^*.$

$x=xvawx=x(vawx)^*=x[v(aw)^d(aw)^2x]^*=x(wawx)^*[v(aw)^da]^*$, and so $x\in x\mathcal{A}(va^{d,w})^*.$

Therefore $a_{[v,w]}^{\tiny\textcircled{d}}=(vaw)^{\big(a^{d,w}, (va^{d,w})^*\big)}.$

$(2)\Rightarrow (1)$ Let $x=(vaw)^{\big(a^{d,w}, (va^{d,w})^*\big)}$. Then
$$\begin{array}{c}
xvawx=x, x(vaw)(aw)^da=(aw)^da, (v(aw)^da)^*(vaw)x=(v(aw)^da)^*,\\
x\in (aw)^da\mathcal{A}x\bigcap x\mathcal{A}(v(aw)^da)^*.
\end{array}$$ By assumption, we have $(vawx)^*=vawx$.

Claim 1. $x\mathcal{A}=a^{d,w}\mathcal{A}$. Obviously, $x\in a^{d,w}\mathcal{A}$. On the other hand, we have
$$a^{d,w}=[(aw)^d]^2a=(aw)^d(aw)[(aw)^d]^2a\in x\mathcal{A}.$$

Claim 2. $\mathcal{A}x=\mathcal{A}(va^{d,w})^*$. We verify that

$$x=xvawx=x(vawx)^*=x(v[(aw)^d]^2(aw)^2x)=x(wawx)^*(va^{d,w})^*,$$ and so $x\in \mathcal{A}(va^{d,w})^*$.

Since $(v(aw)^da)^*(vaw)x=(v(aw)^da)^*$, we see that $$(va^{d,w})^*=[v(aw)^daw((aw)^d)^2a]^*\in \mathcal{A}(a^{d,w})^*\subseteq \mathcal{A}x.$$
Therefore $\mathcal{A}x=\mathcal{A}(va^{d,w})^*$, as asserted.\end{proof}

\begin{cor} Let $a\in \mathcal{A}^{d}$ and let $v\in \mathcal{A}^{-1}$ be Hermitian. Then the following are equivalent:\end{cor}
\begin{enumerate}
\item [(1)] $a\in \mathcal{A}^{v,\tiny\textcircled{d}}$.
\item [(2)] $a\in \mathcal{A}^{\big(a^{d}, (va^{d})^*\big)}$ and $aa^{\big(a^{d}, (va^{d})^*\big)}$ is Hermitian.
\end{enumerate} In this case, $a^{v,\tiny\textcircled{d}}=a^{\big(a^{d}, (va^{d})^*\big)}.$
\begin{proof} This is obvious by choosing $w=1$ in Theorem 3.7.\end{proof}

\section{representations of the generalized $[v,w]$-core inverse}

The aim of this section is to establish representations of the generalized $[v,w]$-core inverse with an invertible Hermitian weight $v$.

\begin{thm} Let $a\in \mathcal{A}^{d,w}$ and $v\in \mathcal{A}^{-1}$ be Hermitian. Then $a\in \mathcal{A}_{[v,w]}^{\tiny\textcircled{d}}$ if and only if $a^{d,w}\in\mathcal{A}_{[v,w]}^{\tiny\textcircled{\#}}$. In this case, $$a_{[v,w]}^{\tiny\textcircled{d}}=[a^{d,w}w]^2(a^{d,w})_{[v,w]}^{\tiny\textcircled{\#}}.$$\end{thm}
\begin{proof} $\Longrightarrow $ Let $x=a_{[v,w]}^{\tiny\textcircled{d}}$. Then we can find
$x\in \mathcal{A}$ such that $$\begin{array}{c}
x=awxvx=xvawx, (vawx)^*=vawx,\\
\lim\limits_{n\to \infty}||(aw)^n-(xv)(aw)^{n+1}||^{\frac{1}{n}}=0.
\end{array}$$
We verify that
 $$\begin{array}{rll}
||(aw)^d-xvaw(aw)^d||&=&||(aw)^n[(aw)^d]^{n+1}-xv(aw)^{n+1}[(aw)^d]^{n+1}||\\
&\leq&||(aw)^n-xv(aw)^{n+1}||||(aw)^d]^{n+1}||.
\end{array}$$ Since $\lim\limits_{n\to \infty}||(aw)^n-xv(aw)^{n+1}||^{\frac{1}{n}}=0$, we deduce that $$\lim\limits_{n\to \infty}||(aw)^d-xv(aw)(aw)^d||^{\frac{1}{n}}=0.$$
Hence, $xv(aw)(aw)^d=(aw)^d$. Let $z=(aw)^2x$. Then
$$\begin{array}{rll}
a^{d,w}wzvz&=&(aw)^dzvz=(aw)^d(aw)^2xvz=awxvz=z,\\
zva^{d,w}wz&=&zv(aw)^dz=(aw)^2xv(aw)^dz=aw(aw)^dz=z,\\
va^{d,w}wz&=&v(aw)^d(aw)^2x=vawx,\\
(va^{d,w}wz)^*&=&(vawx)^*=vawx=va^{d,w}wz,\\
a^{d,w}wzv(a^{d,w}w)a^{d,w}&=&(aw)^dzv(aw)^da^{d,w}=(aw)^d(aw)^2xv(aw)[(aw)^d]^2a^{d,w}\\
&=&(aw)(xv)(aw)[(aw)^d]^4a =[(aw)^d]^3a=a^{d,w}wa^{d,w},\\
zv(a^{d,w}w)^2a^{d,w}&=&(aw)^2xv((aw)^d)^4a=(aw)((aw)^d)^4a=a^{d,w}wa^{d,w}.
\end{array}$$ Thus, $a^{d,w}\in \mathcal{A}_{[v,w]}^{\tiny\textcircled{\#}}$ and $(a^{d,w})_{[v,w]}^{\tiny\textcircled{\#}}=z=(aw)^2a_{[v,w]}^{\tiny\textcircled{d}},$ as desired.

$\Longleftarrow $ Set $x=[a^{d,w}w]^2(a^{d,w})_{[v,w]}^{\tiny\textcircled{\#}}$. Let $z=(a^{d,w})_{[v,w]}^{\tiny\textcircled{\#}}.$
Then we have $$\begin{array}{c}
a^{d,w}wzvz=zva^{d,w}wz=z, (va^{d,w}wz)^*=va^{d,w}wz,\\
a^{d,w}wzv(a^{d,w}w)a^{d,w}=zv(a^{d,w}w)^2a^{d,w}=a^{d,w}wa^{d,w}.
\end{array}$$ Hence,

$$\begin{array}{c}
(aw)^dzvz=zv(aw)^dz=z, (v(aw)^dz)^*=v(aw)^dz,\\
(aw)^dzv(aw)^da^{d,w}=zv((aw)^d)^2a^{d,w}=(aw)^da^{d,w}.
\end{array}$$
 Then $x=[(aw)^d]^2z$. We verify that
$$\begin{array}{rll}
awxvx&=&aw[(aw)^d]^2zv[(aw)^d]^2z\\
&=&(aw)^d\big(zv[(aw)^d]^2\big)z\\
&=&(aw)^d\big(zv[(aw)^d]^2\big)z\\
&=&(aw)^d\big(zv[(aw)^d]^2a^{d,w}\big)[w(aw)z]\\
&=&[(aw)^d]^2z=x,\\
xvawx&=&[(aw)^d]^2zvaw[(aw)^d]^2z=[(aw)^d]^2[zv(aw)^dz]=[(aw)^d]^2z=x,\\
vawx&=&vaw[(aw)^d]^2z=v(aw)^dz,\\
(vawx)^*&=&[v(aw)^dz]^*=v(aw)^dz=vawx,
\end{array}$$

Observing that $$\begin{array}{rl}
&(aw)^n-(xv)(aw)^{n+1}\\
=&(aw)^n-(xv)(aw)^d(aw)^{n+2}\\
=&(aw)^n-[(aw)^d]^2zv(aw)^d(aw)^{n+2}\\
=&(aw)^n-[(aw)^d]^2zv(aw)^d[(aw)^d]^2a[w(aw)^{n+3}]\\
=&(aw)^n-[(aw)^d][(aw)^da^{d,w}][w(aw)^{n+3}]\\
=&(aw)^n-[(aw)^d]^3[(aw)^{n+3}]\\
=&(aw)^n-(aw)^d[(aw)^{n+1},
\end{array}$$ we have $||(aw)^n-(xv)(aw)^{n+1}|||^{\frac{1}{n}}=0.$ Accordingly, $$a_{[v,w]}^{\tiny\textcircled{d}}=[a^{d,w}w]^2(a^{d,w})_{[v,w]}^{\tiny\textcircled{\#}},$$ as required.\end{proof}

Let $a,x\in \mathcal{A}$. The element $x$ is called $(1,3,v)$-inverse of $a$ and is denoted by $a^{(1,3,v)}$ provided that $axva=a$ and $(vax)^*=vax$.
We use $\mathcal{A}^{(1,3,v)}$ to stand for the set of all $(1,3,v)$ invertible elements in $\mathcal{A}$.

\begin{cor} Let $a\in \mathcal{A}^{d,w}$ and $v\in \mathcal{A}^{-1}$ be Hermitian. Then the following are equivalent:\end{cor}
\begin{enumerate}
\item [(1)] $a\in \mathcal{A}_{[v,w]}^{\tiny\textcircled{d}}$.
\vspace{-.5mm}
\item [(2)] $a^{d,w}\in \mathcal{A}^{(1,3,v)}$.
\end{enumerate}
In this case, $a_{[v,w]}^{\tiny\textcircled{d}}=a^{d,w}wa^{d,w}(a^{d,w})^{(1,3,v)}.$
\begin{proof} $(1)\Rightarrow (2)$ By virtue of Theorem 4.1, $a^{d,w}\in \mathcal{A}_{[v,w]}^{\tiny\textcircled{\#}}$.
Set $z=(a^{d,w})_{[v,w]}^{\tiny\textcircled{\#}}$. Then
$a^{d,w}wzva^{d,w}=a^{d,w}.$ Hence $a^{d,w}wzva^{d,w}=a^{d,w}$.
Since $(va^{d,w}w)^*=va^{d,w}w$, we prove that $a^{d,w}\in \mathcal{A}^{(1,3,v)}$, as desired.

$(2)\Rightarrow (1)$ Let $x=a^{d,w}wa^{d,w}(a^{d,w})^{(1,3,v)}$. Then we check that
$$\begin{array}{rll}
awxvx&=&awa^{d,w}wa^{d,w}(a^{d,w})^{(1,3,v)}va^{d,w}wa^{d,w}(a^{d,w})^{(1,3,v)}\\
&=&[a^{d,w}(a^{d,w})^{(1,3,v)}va^{d,w}]wa^{d,w}(a^{d,w})^{(1,3,v)}\\
&=&a^{d,w}wa^{d,w}(a^{d,w})^{(1,3,v)}=x.
\end{array}$$
Obviously, we have $x\mathcal{A}\subseteq a^{d,w}\mathcal{A}$. On the other hand,
$$a^{d,w}=[a^{d,w}wa^{d,w}]wa=[a^{d,w}wa^{d,w}(a^{d,w})^{(1,3,v)}]a^{d,w}wa=xa^{d,w}wa;$$ whence, $a^{d,w}\mathcal{A}\subseteq x\mathcal{A}.$
Hence, $x\mathcal{A}=a^{d,w}\mathcal{A}$.

We easily verify that $$\begin{array}{rll}
x&=&a^{d,w}wa^{d,w}(a^{d,w})^{(1,3)}=a^{d,w}w[a^{d,w}(a^{d,w})^{(1,3)}]\\
&=&a^{d,w}wv^{-1}[va^{d,w}(a^{d,w})^{(1,3)}]^*=a^{d,w}wv^{-1}[(a^{d,w})^{(1,3)}]^*[va^{d,w}]^*;
\end{array}$$ and so
$\mathcal{A}x\subseteq \mathcal{A}(va^{d,w})^*$. On the other hand, we have $$\begin{array}{rll}
[va^{d,w}]^*&=&[va^{d,w}(a^{d,w})^{(1,3)}va^{d,w}]^*=[(va^{d,w}(a^{d,w})^{(1,3)})^*(va^{d,w})]^*\\
&=&[va^{d,w}]^*va^{d,w}(a^{d,w})^{(1,3)}\\
&=&[va^{d,w}]^*[va^{d,w}wawa^{d,w}](a^{d,w})^{(1,3)}\\
&=&[a^{d,w}]^*[a^{d,w}wa]wa^{d,w}(a^{d,w})^{(1,3)}\\
&=&[a^{d,w}]^*[awa^{d,w}]wa^{d,w}(a^{d,w})^{(1,3)}\\
&=&[a^{d,w}]^*aw[a^{d,w}wa^{d,w}(a^{d,w})^{(1,3)}]\\
&=&[a^{d,w}]^*awx,
\end{array}$$ and then $\mathcal{A}(va^{d,w})^*\subseteq \mathcal{A}x$.Thus $\mathcal{A}x=\mathcal{A}(va^{d,w})^*$.
Therefore $a\in \mathcal{A}_{[v,w]}^{\tiny\textcircled{d}}$ by Theorem 4.1.\end{proof}

\begin{cor} Let $a\in \mathcal{A}^{d,w}$ and $v\in \mathcal{A}^{-1}$ be Hermitian. Then the following are equivalent:\end{cor}
\begin{enumerate}
\item [(1)] $a\in \mathcal{A}_{[v,w]}^{\tiny\textcircled{d}}$.
\vspace{-.5mm}
\item [(2)] There exists an idempotent $q\in \mathcal{A}$ such that $(v^2q)^*=v^2q$ and $a^{d,w}\mathcal{A}=q\mathcal{A}$.
\end{enumerate}
\begin{proof}  $(1)\Rightarrow (2)$ In view of Corollary 4.2, we have $a^{d,w}\in \mathcal{A}^{(1,3,v)}$. Then
$a^{d,w}=a^{d,w}(a^{d,w})^{(1,3,v)}va^{d,w}$ and $[va^{d,w}(a^{d,w})^{(1,3)}]^*=va^{d,w}(a^{d,w})^{(1,3)}$. Let
$q=a^{d,w}(a^{d,w})^{(1,3)}_vv$. Then $q^2=q$ and $a^{d,w}\mathcal{A}=q\mathcal{A}$.
Since $(vqv^{-1})^*=vqv^{-1}$, we have $(v^2q)^*=v^2q$, as required.

$(2)\Rightarrow (1)$ By hypothesis, there exists an idempotent $q\in \mathcal{A}$ such that $a^{d,w}\mathcal{A}=q\mathcal{A}$ and $(v^2q)^*=v^2q$.
Write $a^{d,w}=qx$ and $q=a^{d,w}yv^2$ for some $x,y\in \mathcal{A}$. Then $a^{d,w}(yv)va^{d,w}=qa^{d,w}=a^{d,w}$.
Since $(v^2q)^*=v^2q$, we see that $(vqv^{-1})^*=vqv^{-1}$. Thus $(va^{d,w}yv)^*=va^{d,w}yv$.
This implies that $a^{d,w}\in \mathcal{A}^{(1,3,v)}$. Therefore we obtain the rest by Corollary 4.2.\end{proof}

\begin{thm} Let $a\in \mathcal{A}^{d,w}$ and $v\in \mathcal{A}^{-1}$ be Hermitian. Then the following are equivalent:\end{thm}
\begin{enumerate}
\item [(1)] $a\in \mathcal{A}_{[v,w]}^{\tiny\textcircled{d}}$.
\vspace{-.5mm}
\item [(2)] $awa^{d,w}w\in \mathcal{A}^{(1,3,v)}$.
\end{enumerate}
In this case, $a_{[v,w]}^{\tiny\textcircled{d}}=a^{d,w}w(awa^{d,w}w)^{(1,3,v)}.$
\begin{proof} $(1)\Rightarrow (2)$ Set $x=a_{[v,w]}^{\tiny\textcircled{d}}$. Then
$$\begin{array}{rll}
||a^{d,w}w-xv(awa^{d,w}w||&=&||(aw)^d-xv(aw)(aw)^d||\\
&\leq&||(aw)^m-xv(aw)^{m+1}||||((aw)^d)^{m+1}||.
\end{array}$$ Since  $$\lim\limits_{m\to \infty}||(aw)^m-xv(aw)^{m+1}||^{\frac{1}{m}}=0,$$ we have
$$\lim\limits_{m\to \infty}||a^{d,w}w-xv(awa^{d,w})w||^{\frac{1}{m}}=0.$$
Hence $xvawa^{d,w}w=a^{d,w}w,$ and then $$\begin{array}{rll}
[awa^{d,w}w][awx]&=&aw[(aw)^d]^2awawx=awx,\\
(v(awa^{d,w}w)(awx))^*&=&(vawx)^*=vawx=v(awa^{d,w}w)(awx),\\
(awa^{d,w}w)(awx)v(awa^{d,w}w)&=&(awx)v(awa^{d,w}w)=(awxvaw)a^{d,w}w\\
&=&awa^{d,w}w.
\end{array}$$ Accordingly, $awa^{d,w}w\in \mathcal{A}^{(1,3,v)}$, as desired.

$(2)\Rightarrow (1)$ Let $x=a^{d,w}w(awa^{d,w}w)^{(1,3,v)}$. Then we verify that
$$\begin{array}{rll}
awx&=&awa^{d,w}w(awa^{d,w}w)^{(1,3,v)}=aw(aw)^d(aw(aw)^d)^{(1,3,v)},\\
(vawx)^*&=&vawx,\\
awxvx&=&(awx)vx=aw(aw)^d(aw(aw)^d)^{(1,3,v)}va^{d,w}w(awa^{d,w}w)^{(1,3,v)}\\
&=&aw(aw)^d(aw(aw)^d)^{(1,3,v)}v(aw)^d(awa^{d,w}w)^{(1,3,v)}\\
&=&aw(aw)^d(aw(aw)^d)^{(1,3,v)}v[(aw)(aw)^d](aw)^d(awa^{d,w}w)^{(1,3,v)}\\
&=&aw[(aw)^d]^2(awa^{d,w}w)^{(1,3,v)}\\
&=&(aw)^d(awa^{d,w}w)^{(1,3,v)}=x.
\end{array}$$

Clearly, we have $$\begin{array}{rl}
&[1-(xv)(aw)](aw)^n\\
=&[1-aw(aw)^d(aw(aw)^d)^{(1,3,v)}v](aw)^n\\
=&[1-aw(aw)^d(aw(aw)^d)^{(1,3,v)}v(aw)(aw)^d](aw)^n\\
-&aw(aw)^d(aw(aw)^d)^{(1,3,v)}v[(aw)^n-(aw)^d(aw)^{n+1}]\\
=&[1-aw(aw)^d](aw)^n-aw(aw)^d(aw(aw)^d)^{(1,3,v)}v[(aw)^n-(aw)^d(aw)^{n+1}],
\end{array}$$ and then
$$\begin{array}{rl}
&||(aw)^n-(xv)(aw)^{n+1}||^{\frac{1}{n}}\\
\leq&||1-aw(aw)^d(aw(aw)^d)^{(1,3,v)}v||^{\frac{1}{n}}||(aw)^n-(aw)^d(aw)^{n+1}||^{\frac{1}{n}}.\\
\end{array}$$ Hence, $$\lim\limits_{n\to \infty}||(aw)^n-(xv)(aw)^{n+1}||^{\frac{1}{n}}=0.$$
Accordingly, $a\in \mathcal{A}_{[v,w]}^{\tiny\textcircled{d}}$. In this case, $a_{[v,w]}^{\tiny\textcircled{d}}=a^{d,w}w(awa^{d,w}w)^{(1,3,v)}.$\end{proof}

\begin{cor} Let $a,w\in \mathcal{A}$. Then the following are equivalent:\end{cor}
\begin{enumerate}
\item [(1)] $a\in \mathcal{A}_w^{\tiny\textcircled{d}}$.
\vspace{-.5mm}
\item [(2)] $a\in \mathcal{A}^{d,w}$ and $(aw)(aw)^d\in \mathcal{A}^{(1,3)}$.
\end{enumerate}
In this case, $a^{\tiny\textcircled{d}}=(aw)^{d}[aw(aw)^{d}]^{(1,3)}.$
\begin{proof} This is obvious by choosing $v=1$ in Theorem 4.4.\end{proof}

\section{applications to complex matrices}

This section is devoted to present applications in the complex setting. Let $\mathcal{R}(X)$ represent the range space of a complex matrix $X$ and $X^*$ be the conjugate transpose of $X$. For further use, we derive

\begin{lem} Let $A,W\in {\Bbb C}^{n\times n}$ and let $V\in {\Bbb C}^{n\times n}$ be invertible. Then $A$ is $(W,V)$-pseudo core invertible if and only if $A$ has generalized $[W,V]$-core inverse. In this case, $A_{[V,W]}^{\tiny\textcircled{D}}=A_{[V,W]}^{\tiny\textcircled{d}}$.\end{lem}
\begin{proof} $(1)\Rightarrow (1)$ Set $X=A_{[V,W]}^{\tiny\textcircled{D}}$. Then we have
$$\begin{array}{c}
AWXVX=XVAWX=X, (VAWX)^*=VAWX,\\
XV(AW)^{k+1}A=(AW)^kA, AWXV(AW)^kA=(AW)^kA
\end{array}$$ for some $k\in {\Bbb N}$. Then we verify that
$$\begin{array}{rll}
XV(AW)^{k+2}&=&[XV(AW)^{k+1}A]W\\
&=&[(AW)^{k}A]W\\
&=&(AW)^{k+1}.
\end{array}$$ Therefore $A_{[V,W]}^{\tiny\textcircled{d}}=X$, as required.

$(2)\Rightarrow (1)$ Set $X=A_{[V,W]}^{\tiny\textcircled{d}}$. In view of Theorem 3.5,  we have
$$AWXVX=X, (VAWX)^*=VAWX, R(X)=R(A^{D,W}), R(X)=R(VA^{D,W})^*.$$ Set $k=ind(AW)$. Then we have
$(AW)^k=(AW)^D(AW)^{k+1}$. Hence, we derive that
$$A^{D,W}=[(AW)^D]^2A=[(AW)^D]^2[(AW)^D]^k(AW)^{k}A,$$ and then
$$R\big(A^{D,W}\big)=R\big((AW)^kA\big).$$
According to ~\cite[Theorem 2.17]{M}, $A_{[V,W]}^{\tiny\textcircled{D}}=X$, as desired.\end{proof}

We now proceed to prove:

\begin{thm} Let $A,W\in {\Bbb C}^{n\times n}$ and let $V\in {\Bbb C}^{n\times n}$ be invertible. Then the following are equivalent:\end{thm}
\begin{enumerate}
\item [(1)] $A_{[V,W]}^{\tiny\textcircled{D}}=X$.
\vspace{-.5mm}
\item [(2)] $A=Z+Y, YWZ=0, Z\in \mathcal{A}_{[V,W]}^{\tiny\textcircled{\#}}, YW~\mbox{is nilpotent}.$
\vspace{-.5mm}
\item [(3)] There exists an idempotent $P\in {\Bbb C}^{n\times n}$ such that
$$AW+P~\mbox{is invertible}, (V^2P)^*=V^2P, P(AW)=P(AW)P~\mbox{is nilpotent}.$$
\item [(4)] $XVAWX=X, (VAWX)^*=VAWX,\mathcal{R}(X)=\mathcal{R}(X^*)=\mathcal{R}\big(V(AW)^D\big).$
\vspace{-.5mm}
\item [(5)] $A^{D,W}\in [{\Bbb C}^{n\times n}]_{[V,W]}^{\tiny\textcircled{\#}}$ and $[A^{D,W}W]^2(A^{D,W})_{[V,W]}^{\tiny\textcircled{\#}}=X$.
\vspace{-.5mm}
\item [(6)] $AWA^{D,W}W\in [{\Bbb C}^{n\times n}]^{(1,3,V)}$ and $A^{D,W}W(AWA^{D,W}W)^{(1,3,V)}=X$.
\end{enumerate}
\begin{proof} As is well known, the quasinilpotent and nilpotent complex matrix coincide with each other.
Therefore we obtain the result by Lemma 5.1, Theorem 2.2, Theorem 3.2, Theorem 4.1 and Theorem 4.4.\end{proof}

\begin{exam}\end{exam} Let $A=\left(
  \begin{array}{ccc}
    1&1&0\\
    0&0&0\\
    1&1&-1\\
  \end{array}
\right), V=\left(
  \begin{array}{cccc}
    1&0&0\\
    0&-1&1\\
    1&1&1\\
  \end{array}
\right), W=\left(
  \begin{array}{cccc}
    1&1&0\\
    0&-1&0\\
    1&0&0\\
  \end{array}
\right)\in {\Bbb C}^{3\times 3}$. We take the involution on ${\Bbb C}^{4\times 4}$ as the conjugate transpose. Set $X=\left(
\begin{array}{ccc}
    1&0&0\\
    0&0&0\\
    0&0&0\\
  \end{array}
\right)$. Then we check that
$$\begin{array}{c}
AWXVX=X, (VAWX)^*=VAWX,\\
XV(AW)^2XV=AWXV, AW=AWXVAW.
\end{array}$$
Therefore $$A_{[V,W]}^{\tiny\textcircled{D}}=\left(
  \begin{array}{ccc}
    1&0&0\\
    0&0&0\\
    0&0&0\\
  \end{array}
\right).$$ Choose $$Z=AWXVA=\left(
  \begin{array}{ccc}
    1&1&0\\
    0&0&0\\
    0&0&0\\
  \end{array}
\right), Y=A-AWXVA=\left(
  \begin{array}{ccc}
    0&0&0\\
    0&0&0\\
    1&1&-1\\
  \end{array}
\right).$$
Evidently, $$\begin{array}{c}
A=Z+Y, (VZ)^*VY=YWZ=0, YW=0,\\
Z_{[V,W]}^{\tiny\textcircled{\#}}=\left(
  \begin{array}{ccc}
    1&0&0\\
    0&0&0\\
    0&0&0\\
  \end{array}
\right).
\end{array}$$

\begin{thm} Let $A,W\in {\Bbb C}^{n\times n}$, where $(AW)^k=BC$ is a full-rank decomposition, $k=ind(AW)$ and $V$ is a Hermitian positive definite matrix.
Then $$A_{[V,W]}^{\tiny\textcircled{D}}=B(CB)^{-2}CV^{-\frac{1}{2}}[V^{\frac{1}{2}}\big(B(CB)^{-2}C\big)^{k-1}V^{-\frac{1}{2}}]^{\dag}V^{\frac{1}{2}}.$$
\end{thm}
\begin{proof} Set $k=ind(AW)$. Then $(AW)^k$ has group inverse, and so $CB$ is invertible. Hence,
we have $(AW)^D=\big((AW)^k\big)^{\#}=(BC)^{\#}=B(CB)^{-2}C$. Since $V$ is a Hermitian positive definite matrix,
it follows by ~\cite[Theorem 5]{KD} that
$$[(AW)(AW)^D]_{V,V}^{\dag}=V^{-\frac{1}{2}}[V^{\frac{1}{2}}AW(AW)^DV^{-\frac{1}{2}}]^{\dag}V^{\frac{1}{2}}.$$
By virtue of Theorem 5.2, we have
$$\begin{array}{rll}
A_{[V,W]}^{\tiny\textcircled{D}}&=&A^{D,W}W(AWA^{D,W}W)^{(1,3,V)}\\
&=&(AW)^{D}[AW(AW)^D]^{(1,3,V)}\\
&=&(AW)^D[(AW)(AW)^D]_{V,V}^{\dag}\\
&=&(AW)^DV^{-\frac{1}{2}}[V^{\frac{1}{2}}AW(AW)^DV^{-\frac{1}{2}}]^{\dag}V^{\frac{1}{2}}\\
&=&B(CB)^{-2}CV^{-\frac{1}{2}}[V^{\frac{1}{2}}(AW)^k\big((AW)^D\big)^kV^{-\frac{1}{2}}]^{\dag}V^{\frac{1}{2}}\\
&=&B(CB)^{-2}CV^{-\frac{1}{2}}[V^{\frac{1}{2}}BC\big(B(CB)^{-2}C\big)^kV^{-\frac{1}{2}}]^{\dag}V^{\frac{1}{2}}\\
&=&B(CB)^{-2}CV^{-\frac{1}{2}}[V^{\frac{1}{2}}\big(B(CB)^{-2}C\big)^{k-1}V^{-\frac{1}{2}}]^{\dag}V^{\frac{1}{2}}.
\end{array}$$ as asserted.\end{proof}

\begin{exam}\end{exam} Let $A=\left(
  \begin{array}{ccc}
    1&i&0\\
    0&0&1\\
    0&0&0\\
  \end{array}
\right), V=\left(
  \begin{array}{ccc}
    1&0&0\\
    0&4&0\\
    0&0&1\\
  \end{array}
\right), W=\left(
  \begin{array}{ccc}
    1&i&0\\
    0&1&0\\
    0&0&1\\
  \end{array}
\right)\in {\Bbb C}^{3\times 3}$. We take the involution on ${\Bbb C}^{3\times 3}$ as the conjugate transpose. Then $ind(AW)=2$.
Then $$AW=\left(
  \begin{array}{ccc}
    1&i&0\\
    0&0&1\\
    0&0&0\\
  \end{array}
\right).$$ Clearly, we have a full-rank decomposition of $(AW)^2$:
$$(AW)^2=\left(
  \begin{array}{ccc}
    1&i&i\\
    0&0&0\\
    0&0&0\\
  \end{array}
\right)=BC, B=\left(
\begin{array}{c}
1\\
0\\
0
\end{array}
\right), C=(1,i,i).$$

$$V^{\frac{1}{2}}=\left(
  \begin{array}{ccc}
    1&0&0\\
    0&2&0\\
    0&0&1\\
  \end{array}
\right), V^{-\frac{1}{2}}=\left(
  \begin{array}{ccc}
    1&0&0\\
    0&\frac{1}{2}&0\\
    0&0&1\\
  \end{array}
\right).$$

$$\begin{array}{rl}
&V^{\frac{1}{2}}\big(B(CB)^{-2}C\big)^{k-1}V^{-\frac{1}{2}}\\
=&\left(
  \begin{array}{cccc}
    1&0&0\\
    0&2&0\\
    0&0&1\\
  \end{array}
\right)\left(
  \begin{array}{cccc}
    1&i&i\\
    0&0&0\\
    0&0&0\\
  \end{array}
\right)\left(
  \begin{array}{cccc}
    1&0&0\\
    0&\frac{1}{2}&0\\
    0&0&1\\
  \end{array}
\right)\\
=&\left(
  \begin{array}{cccc}
    1&\frac{1}{2}i&i\\
    0&0&0\\
    0&0&0\\
  \end{array}
\right).
\end{array}$$
Then we compute that $$\left(
  \begin{array}{cccc}
    1&\frac{1}{2}i&i\\
    0&0&0\\
    0&0&0\\
  \end{array}
\right)^{\dag}=\left(
  \begin{array}{cccc}
    \frac{4}{9}&0&0\\
    -\frac{2}{9}i&0&0\\
    -\frac{4}{9}i&0&0\\
  \end{array}
\right).$$
By virtue of Theorem 5.4, we have
$$\begin{array}{rll}
A_{[V,W]}^{\tiny\textcircled{D}}&=&\left(
  \begin{array}{cccc}
    1&i&i\\
    0&0&0\\
    0&0&0\\
  \end{array}
\right)\left(
  \begin{array}{ccc}
    1&0&0\\
    0&\frac{1}{2}&0\\
    0&0&1\\
  \end{array}
\right)\left(
  \begin{array}{cccc}
    \frac{4}{9}&0&0\\
    -\frac{2}{9}i&0&0\\
    -\frac{4}{9}i&0&0\\
  \end{array}
\right)\left(
  \begin{array}{ccc}
    1&0&0\\
    0&2&0\\
    0&0&1\\
  \end{array}
\right)\\
&=&\left(
  \begin{array}{cccc}
    \frac{5-2i}{9}&0&0\\
    0&0&0\\
    0&0&0\\
  \end{array}
\right).
\end{array}$$

\vskip10mm
{\bf Disclosure statements}\\

No potential conflict of interest was reported by the authors.\\

\end{document}